\documentclass[10pt,a4paper]{amsart}
\usepackage[a4paper,marginratio=1:1]{geometry}
\usepackage[initials,non-sorted-cites]{amsrefs}
\usepackage{mybibspec}
\usepackage[shortlabels]{enumitem}
\usepackage[symbol]{footmisc}
\usepackage{braket,tensor}

\numberwithin{equation}{section}

\newcommand{\bdry}{\partial}
\newcommand{\conj}[1]{\overline{#1}}

\DeclareMathOperator{\Ric}{Ric}
\newcommand{\SU}{\mathit{SU}}
\newcommand{\Sp}{\mathit{Sp}}

\theoremstyle{theorem}
\newtheorem{thm}{Theorem}[section]

\newtheorem{lem}[thm]{Lemma}

\theoremstyle{definition}

\theoremstyle{remark}

\title[]{A construction of Poincar\'e-Einstein metrics of cohomogeneity one on the ball}
\author{Yoshihiko Matsumoto}
\address{Department of Mathematics, Graduate School of Science, Osaka University, Toyonaka, Osaka 560-0043, Japan}
\email{matsumoto@math.sci.osaka-u.ac.jp}
\address{Department of Mathematics, Stanford University, Stanford, CA 94305-2125, USA}

\begin{document}

\begin{abstract}
	We exhibit an explicit one-parameter smooth family of Poincar\'e-Einstein metrics on the
	even-dimensional unit ball whose conformal infinities are the Berger spheres.
	Our construction is based on a Gibbons--Hawking-type ans\"atz of Page and Pope.
	The family contains the hyperbolic metric, converges to the complex hyperbolic metric at one of the ends,
	and at the other end the ball equipped with our metric
	collapses to a Poincar\'e-Einstein manifold of one lower dimension with an isolated conical singularity.
\end{abstract}

\maketitle

\section{Introduction}
\label{sec:introduction}

Existence theory of Poincar\'e-Einstein metrics awaits further progress.
In this article, we focus on those on the unit ball $B^{2n}\subset\mathbb{C}^n$
invariant under the standard action of $\SU(n)$ for $n\ge 2$,
and by the general technique of Page--Pope \cite{Page-Pope-87}
we construct a one-parameter family of such metrics.
As the generic orbits of the $\SU(n)$-action are hypersurfaces of $B^{2n}$,
those metrics are said to be of cohomogeneity one in the literature.

Our metrics will be denoted by $g_c$ with a parameter $c$ running through the interval $(0,\infty)$,
and $g_1$ is the hyperbolic metric (the Poincar\'e metric).
The conformal infinity of $g_c$ is the manifold $(S^{2n-1},[h_c])$, where $h_c$ denotes the Berger-type metric
on the sphere $S^{2n-1}$:
\begin{equation}
	\label{eq:Berger-metric}
	h_c=c\theta^2+g_{\mathbb{C}P^{n-1}}.
\end{equation}
Here $\theta$ is the standard contact form,
which is seen here as a connection form of the Hopf fibration $S^{2n-1}\to\mathbb{C}P^{n-1}$,
and $g_{\mathbb{C}P^{n-1}}$ denotes the Fubini--Study metric on $\mathbb{C}P^{n-1}$ lifted to
the horizontal distribution over $S^{2n-1}$.
We normalize $g_{\mathbb{C}P^{n-1}}$ so that its diameter is $\pi/2$ and hence $h_1$ is the round metric.

It is easily observed that any $\SU(n)$-invariant conformal class on $S^{2n-1}$ is
the class $[h_c]$ for some $c\in(0,\infty)$ when $n\ge 3$
(see the beginning of Section \ref{sec:construction-of-metrics}).
Therefore, our family establishes the validity of the following theorem.

\begin{thm}
	\label{thm:SUn-invariant-fillable}
	Let $n\ge 3$.
	Then any $\SU(n)$-invariant conformal class on $S^{2n-1}$ is the conformal infinity
	of an $\SU(n)$-invariant Poincar\'e-Einstein metric on $B^{2n}$.
\end{thm}

When $n=2$, our family recovers the following metrics of Pedersen \cite[Equation (1.6)]{Pedersen-86}
constructed by techniques of twistor theory ($\rho$ denotes the Euclidean distance from the origin of $B^4$,
and we put the factor $4$ in front of the original formula):
\begin{equation}
	\label{eq:Pedersen-metric}
	g=\frac{4}{(1-\rho^2)^2}
	\left(\frac{1+m^2\rho^2}{1+m^2\rho^4}d\rho^2
	+\frac{\rho^2(1+m^2\rho^4)}{1+m^2\rho^2}\theta^2
	+\rho^2(1+m^2\rho^2)g_{\mathbb{C}P^1}\right).
\end{equation}
The parameter $m^2$ should be regarded as a formal symbol which takes any real value
greater than $-1$. Its relation to the parameter $c$ of our family $g_c$ is given by $c^{-1}=1+m^2$.
In this dimension, the set of all $\SU(2)$-invariant conformal classes is not exhausted by
the Berger metrics, and the existence of Poincar\'e-Einstein metrics
for arbitrary $\SU(2)$-invariant conformal infinity is proved by Hitchin \cite[Theorem 8]{Hitchin-95}.

There is a notable observation regarding the limiting behavior of Pedersen's family \eqref{eq:Pedersen-metric}:
It converges to the complex hyperbolic metric as $m^2\to -1$, or equivalently, $c\to\infty$
\cite[Remark (7.2), iv)]{Pedersen-86}.
Interestingly enough, we can show that our family $g_c$ has the same feature for arbitrary $n$.
Let us write, on $B^{2n}$,
\begin{equation}
	\label{eq:complex-hyperbolic-metric}
	g_{\mathbb{C}H^n}=2\left(\frac{d\rho^2+\rho^2\theta^2}{(1-\rho^2)^2}
	+\frac{\rho^2g_{\mathbb{C}P^{n-1}}}{1-\rho^2}\right).
\end{equation}
The complex hyperbolic metric normalized in this way satisfies $\Ric(g_{\mathbb{C}H^n})=-(n+1)g_{\mathbb{C}H^n}$,
whereas $\Ric(g_c)=-(2n-1)g_c$ according to our convention on Poincar\'e-Einstein metrics.
Then we may formulate the convergence result as follows.
Note that each of our metrics $g_c$ retains the freedom that it can be pulled back by any
$\SU(n)$-equivariant diffeomorphism from $\overline{B^{2n}}$ to itself (although there might be more arbitrariness).
In the following theorem, we refer to the operation of choosing one of such pullbacks of $g_c$
as a \emph{normalization} of $g_c$.

\begin{thm}
	\label{thm:limit-c-infinity}
	Let $n\ge 2$. If each $g_c$ is suitably normalized,
	then $g_c$ converges to $\frac{n+1}{2n-1}g_{\mathbb{C}H^n}$ as $c\to\infty$
	in the $C^\infty$ topology on any compact subset of $B^{2n}$.
\end{thm}

Naturally, we are also interested in the other limit.
Note first that, as $c\to 0$, $(S^{2n-1},h_c)$ collapses to $(\mathbb{C}P^{n-1},g_{\mathbb{C}P^{n-1}})$.
The warped product
\begin{equation}
	\label{eq:limit-metric-c-zero}
	g_0=dr^2+\frac{2n}{2n-3}\sinh^2r\cdot g_{\mathbb{C}P^{n-1}}
\end{equation}
is an Einstein metric on $X^*=(0,\infty)\times\mathbb{C}P^{n-1}$,
and this is actually a Poincar\'e-Einstein metric
defined on a collar neighborhood of $\mathbb{C}P^{n-1}$ with conformal infinity $[g_{\mathbb{C}P^{n-1}}]$.
In fact, if we put $e^{-2r}=\frac{1}{4}b_nx^2$ where $b_n=2n/(2n-3)$, then
\eqref{eq:limit-metric-c-zero} turns out to be
\begin{equation*}
	x^{-2}\left(dx^2+\left(1-\tfrac{1}{4}b_nx^2\right)^2g_{\mathbb{C}P^{n-1}}\right),
\end{equation*}
which is the Fefferman--Graham normal form of the Poincar\'e-Einstein metric relative to $g_{\mathbb{C}P^{n-1}}$
(see \cite[Equation (7.13)]{Fefferman-Graham-12}).
The space $X^*$ equipped with the metric \eqref{eq:limit-metric-c-zero}
has a conical singularity at $r=0$ (unless $n=2$, in which case the apparent singularity $r=0$ is just removable).
One can prove that $(B^{2n},g_c)$ collapses to $(X^*,g_0)$ as $c\to 0$ in the sense of the theorem below.
Let $\tilde{X}=\mathbb{R}^{2n}$ and $\tilde{X}^*=\mathbb{R}^{2n}\setminus\set{o}$,
the latter being identified with $(0,\infty)\times S^{2n-1}$ by the polar coordinates.
We define $p\colon\tilde{X}^*\to X^*$ by $p(r,q)=(r,\varphi(q))$
in terms of the Hopf mapping $\varphi\colon S^{2n-1}\to\mathbb{C}P^{n-1}$.
Let $B_R$ (resp.~$B_{R_1,R_2}$) be the subset $\set{0<r<R}$ (resp.~$\set{R_1<r<R_2}$) of $X^*$,
where $R>0$ (resp.~$R_2>R_1>0$).

\begin{thm}
	\label{thm:limit-c-zero}
	For $n\ge 2$, there is a family of $\SU(n)$-equivariant diffeomorphisms
	$\Phi_c\colon \tilde{X}^*\to B^{2n}\setminus\set{o}$ with the following properties:
	\begin{enumerate}[(i)]
		\item
			$\Phi_c$ extends to a homeomorphism $\tilde{X}\to B^{2n}$.
		\item
			Let $g'_c$ be the metric on $X^*$ for which
			$p$ is a Riemannian submersion from $(\tilde{X}^*,\Phi_c^*g_c)$ to $(X^*,g'_c)$.
			Then $g'_c$ converges to $g_0$ as $c\to 0$, in the $C^\infty$ topology on any $B_{R_1,R_2}$.
		\item
			The diameters of the fibers of $p\colon(\tilde{X}^*,\Phi_c^*g_c)\to(X^*,g'_c)$
			converge to zero as $c\to 0$, uniformly over any $B_R$.
	\end{enumerate}
\end{thm}

We have not yet recalled the definitions of Poincar\'e-Einstein metrics and their conformal infinities.
A brief review on them is given in the beginning of Section \ref{sec:construction-of-metrics},
in which we then construct our metrics $g_c$.
For further background on Poincar\'e-Einstein metrics, we refer the reader to
Fefferman--Graham \cites{Fefferman-Graham-85,Fefferman-Graham-12} and Graham--Lee \cite{Graham-Lee-91},
for example.
Section \ref{sec:page-pope-construction} is devoted to a sketch of the Page--Pope construction,
and in Section \ref{sec:limiting-behavior}, the limiting behavior of the family $g_c$ is investigated.

We remark that our construction of $g_c$ in Section \ref{sec:construction-of-metrics} is also discussed
by Mazzeo and Singer \cite{Mazzeo-Singer-07preprint} to some extent.
More specifically, they use the Page--Pope argument in order to construct Poincar\'e-Einstein metrics,
possibly with singularities, just as we do.
Regarding this part, the novelty of our work lies in the observation that
the metrics so constructed are actually smooth for a right choice of a parameter
(made in \eqref{eq:P-for-SU-invariant-PE} below) involved in the construction.

Cohomogeneity one Poincar\'e-Einstein metrics on the ball
are also recently studied by Li \cites{GangLi-16,GangLi-17,GangLi-18}.
His approach is based on an \emph{a priori} estimate
taking advantage of the volume inequality of Dutta--Javaheri \cite{Dutta-Javaheri-10} and
Li--Qing--Shi \cite{Li-Qing-Shi-17} given in terms of the Yamabe constant of the conformal infinity.
While we have obtained the complete existence result on $\SU(n)$-invariant Poincar\'e-Einstein metrics on $B^{2n}$
(which is discussed in \cites{GangLi-17} to some extent),
the uniqueness aspect in \cite{GangLi-16,GangLi-17,GangLi-18} remains interesting,
and the problem about $\Sp(n)$-invariant metrics on $B^{4n}$ taken up in \cite{GangLi-18} is definitely fascinating.

This work was carried out while the author was working as a visiting scholar at Stanford University,
whose hospitality is gratefully acknowledged.
I am indebted to Rafe Mazzeo for various discussions and for his constant support.
Ryoichi Kobayashi inspired me by his question on the idea of
interpolating the real and the complex hyperbolic metrics by Einstein metrics,
and Jeffrey Case helped me find related literature.
I would also like to thank Kazuo Akutagawa for a valuable comment based on which I could clarified the description.
This work was partially supported by JSPS KAKENHI Grant Number JP17K14189 and JSPS Overseas Research Fellowship.

\section{The Page--Pope construction}
\label{sec:page-pope-construction}

The general construction of Einstein metrics due to Page--Pope \cite{Page-Pope-87} is given in the following setting.
Let $M$ be an $(n-1)$-dimensional complex manifold
equipped with a hermitian line bundle $L$ and a hermitian connection of $L$,
and $C$ be the principal $S^1$-bundle associated to $L$ (which is the bundle of unit vectors in $L$).
Locally, if $\tau\in\mathbb{R}/2\pi\mathbb{Z}$ is the fiber coordinate with respect to
an arbitrarily fixed trivialization of $C$ over an open set $U\subset M$,
then the connection form $\theta$ can be written as
\begin{equation*}
	\theta=d\tau+A,
\end{equation*}
where $A$ is a 1-form that descends to $U$
(note that we adopt the convention in which $\theta$ is an $\mathbb{R}$-valued form
rather than $i\mathbb{R}$-valued). The curvature 2-form is
\begin{equation*}
	F=dA,
\end{equation*}
and this is globally defined on $M$.
We \emph{assume} that $F$ is a K\"ahler $(1,1)$-form associated to a K\"ahler-Einstein metric $\overline{g}$ on $M$,
and set $\Ric(\overline{g})=\lambda\overline{g}$.

Let the variable $r$ run through an open interval $I$ that will be determined later.
On the $2n$-dimensional space $I\times C$, we consider metrics of the form
\begin{equation}
	\label{eq:ansatz}
	g=\alpha(r)^2dr^2+\beta(r)^2\theta^2+\gamma(r)^2\overline{g}
\end{equation}
with positive functions $\alpha$, $\beta$, and $\gamma$.
Let $\theta^a$ ($a=2$, $3$, $\dotsc$, $2n-1$) be an orthonormal local coframe on $(M,\overline{g})$, so that
\begin{equation}
	\label{eq:coframe}
	\eta^0=\alpha\,dr,\qquad
	\eta^1=\beta\theta,\qquad
	\eta^a=\gamma\theta^a
\end{equation}
is an orthonormal coframe on $(I\times C,g)$.
We express the Ricci tensor $R=\Ric(g)$ by using the index notation relative to this coframe.
The tedious computation is carried out in \cite[Section 2]{Page-Pope-87},
according to which all the off-diagonal components vanish, and the diagonal components are given by
\begin{align*}
	\tensor{R}{_0_0}
	&=\left(-\frac{\beta''}{\alpha^2\beta}+\frac{\alpha'\beta'}{\alpha^3\beta}\right)
	+(2n-2)\left(-\frac{\gamma''}{\alpha^2\gamma}+\frac{\alpha'\gamma'}{\alpha^3\gamma}\right),\\
	\tensor{R}{_1_1}
	&=\left(-\frac{\beta''}{\alpha^2\beta}+\frac{\alpha'\beta'}{\alpha^3\beta}\right)
	+(2n-2)\left(-\frac{\beta'\gamma'}{\alpha^2\beta\gamma}+\frac{\beta^2}{\gamma^4}\right),\\
	\tensor{R}{_a_a}
	&=-\frac{\gamma''}{\alpha^2\gamma}+\frac{\alpha'\gamma'}{\alpha^3\gamma}
	-\frac{\beta'\gamma'}{\alpha^2\beta\gamma}-\frac{2\beta^2}{\gamma^4}
	-(2n-3)\left(\frac{\gamma'}{\alpha\gamma}\right)^2+\frac{\lambda}{\gamma^2}.
\end{align*}

Now suppose the Einstein equation $\Ric(g)=\Lambda g$ is satisfied.
Then, since $\tensor{R}{_0_0}=\tensor{R}{_1_1}=\Lambda$, it is necessary that
\begin{equation*}
	-\tensor{R}{_0_0}+\tensor{R}{_1_1}
	=(2n-2)\left(\frac{\gamma''}{\alpha^2\gamma}-\frac{\alpha'\gamma'}{\alpha^3\gamma}
	-\frac{\beta'\gamma'}{\alpha^2\beta\gamma}+\frac{\beta^2}{\gamma^4}\right)=0,
\end{equation*}
or equivalently,
\begin{equation}
	\label{eq:Einstein-eq-transverse}
	\alpha\beta\gamma''-(\alpha'\beta+\alpha\beta')\gamma'+\alpha^3\beta^3\gamma^{-3}=0.
\end{equation}
We here impose an additional assumption that
\begin{equation}
	\label{eq:alpha-beta-constant}
	\alpha\beta=c,\qquad\text{where $c>0$ is a constant}.
\end{equation}
No generality is lost by this, since we have the freedom of reparametrizing $r$
under the ans\"atz \eqref{eq:ansatz}.
Equation \eqref{eq:Einstein-eq-transverse} then becomes
\begin{equation}
	\label{eq:Einstein-eq-transverse-restated}
	\gamma''+c^2\gamma^{-3}=0.
\end{equation}
Integrating \eqref{eq:Einstein-eq-transverse-restated} we get $(\gamma')^2=c^2\gamma^{-2}+C_1$.
If we here set $C_1=c$, then
\begin{equation*}
	\gamma'=\pm\gamma^{-1}\sqrt{c(\gamma^2+c)},
\end{equation*}
and hence
\begin{equation*}
	\frac{\gamma}{\sqrt{\gamma^2+c}}\gamma'=\pm\sqrt{c}.
\end{equation*}
This implies that
\begin{equation}
	\label{eq:gamma-in-r}
	\gamma^2=c(r^2-1)
\end{equation}
is sufficient for \eqref{eq:Einstein-eq-transverse-restated} to hold true.

We can simplify the equation $\tensor{R}{_a_a}=\Lambda$, which is
\begin{equation*}
	-\frac{\gamma''}{\alpha^2\gamma}+\frac{\alpha'\gamma'}{\alpha^3\gamma}
	-\frac{\beta'\gamma'}{\alpha^2\beta\gamma}-\frac{2\beta^2}{\gamma^4}
	-(2n-3)\left(\frac{\gamma'}{\alpha\gamma}\right)^2+\frac{\lambda}{\gamma^2}=\Lambda,
\end{equation*}
by \eqref{eq:alpha-beta-constant}, \eqref{eq:gamma-in-r} and their consequences such as
$\alpha^{-1}\alpha'+\beta^{-1}\beta'=0$, $\gamma'=cr\gamma^{-1}$, and $\gamma''=-c^2\gamma^{-3}$.
Eventually we get
\begin{equation}
	\label{eq:Einstein-eq-tangential}
	-\beta^2-(2n-3)r^2\beta^2-2r(r^2-1)\beta\beta'+\lambda c(r^2-1)=\Lambda c^2(r^2-1)^2.
\end{equation}
Let us now set
\begin{equation}
	\label{eq:beta-given-by-P}
	\beta^2=c^2(r^2-1)^{-n+1}P(r).
\end{equation}
Then one can verify that \eqref{eq:Einstein-eq-tangential} is equivalent (away from $r=0$, $\pm 1$) to
\begin{equation*}
	(r^{-1}P)'=\lambda c^{-1}r^{-2}(r^2-1)^{n-1}-\Lambda r^{-2}(r^2-1)^n.
\end{equation*}
Thus we obtain, suppressing the freedom of choosing the integration constant,
\begin{equation}
	\label{eq:P-as-integral}
	P=r\int(\lambda c^{-1}r^{-2}(r^2-1)^{n-1}-\Lambda r^{-2}(r^2-1)^n)dr.
\end{equation}
Note that $P$ is a polynomial in $r$.

Now the Bianchi identity shows that $R_{00}=\Lambda$ and $R_{11}=\Lambda$ are both satisfied as follows.
We already know that $\Ric(g)$ is expressed as $\Lambda g+\varphi(r)((\eta^0)^2+(\eta^1)^2)$
by some function $\varphi(r)$,
and the contracted second Bianchi identity implies $\varphi(r)\cdot d^*\eta^0=\varphi(r)\cdot d^*\eta^1=0$.
By a straightforward computation we obtain
\begin{equation*}
	d^*\eta^0=\frac{\beta'}{\alpha\beta}+(2n-2)\frac{\gamma'}{\alpha\gamma}
	=\frac{1}{\alpha}(\log \beta\gamma^{2n-2})',
\end{equation*}
and since $(\beta\gamma^{2n-2})^2=c^{2n}(r^2-1)^{n-1}P(r)$ is a polynomial
(and nonzero if $\beta$ is a positive function),
$d^*\eta^0$ does not vanish for all but a finite number of values of $r$.
Hence the function $\varphi(r)$ is identically zero.

Thus we have shown the following result.

\begin{thm}[Page--Pope \cite{Page-Pope-87}]
	Let $I$ be an open interval not intersecting $\set{0,\pm 1}$.
	If $\alpha$, $\beta$, and $\gamma$ are positive functions on $I$ expressed as
	\begin{equation}
		\label{eq:alpha-beta-gamma}
		\alpha^2=(r^2-1)^{n-1}P^{-1},\qquad
		\beta^2=c^2(r^2-1)^{-n+1}P,\qquad
		\gamma^2=c(r^2-1)
	\end{equation}
	with $P$ satisfying \eqref{eq:P-as-integral}, then the metric
	\begin{equation*}
		g=\alpha^2dr^2+\beta^2\theta^2+\gamma^2\overline{g}
	\end{equation*}
	on $I\times C$ satisfies the Einstein equation $\Ric(g)=\Lambda g$.
\end{thm}

In the sequel, we are particularly interested in the case where $I=(r_0,\infty)$ for some $r_0\ge 1$.
Observing the behavior of the polynomial $P$ as $r\to\infty$,
one concludes that one of the following must be satisfied in order for the coefficients
$\alpha$, $\beta$, and $\gamma$ are positive for large $r$:
\begin{enumerate}[(i)]
	\item $\Lambda<0$.
	\item $\Lambda=0$ and $\lambda>0$.
\end{enumerate}
If either is satisfied, then the metric $g$ is always complete ``toward $r=\infty$'' because
\begin{equation*}
	\alpha^2=(r^2-1)^{n-1}P^{-1}\sim
	\begin{cases}
		-(2n-1)^{-1}\Lambda r^{-2}& \text{if $\Lambda<0$},\\
		(2n-3)^{-1}\lambda c^{-1}& \text{if $\Lambda=0$ and $\lambda>0$}.
	\end{cases}
\end{equation*}
The natural choice of the number $r_0$ is
\begin{equation}
	\label{eq:r-zero}
	r_0=\max(\set{1}\cup \Sigma),
	\qquad\text{where $\Sigma=\set{r\in\mathbb{R}|P(r)=0}$}.
\end{equation}
Whether $r=r_0$ is a removable singularity of $g$ or not will be our interest
in the application discussed in the next section.

\section{$\SU(n)$-invariant Poincar\'e-Einstein metrics on the ball}
\label{sec:construction-of-metrics}

Let $\overline{X}$ be a $2n$-dimensional smooth manifold-with-boundary
(hereafter ``smooth'' always means $C^\infty$), whose interior is denoted by $X$.
A Riemannian metric $g$ on $X$ is said to be \emph{smooth conformally compact}
when $x^2g$ extends to a smooth metric on $\overline{X}$ for some,
hence for any, smooth boundary defining function $x$.
When this is the case, the conformal class of $h=(x^2g)|_{T\bdry\overline{X}}$ is
called the \emph{conformal infinity} of $g$.
When a smooth conformally compact Riemannian metric $g$ satisfies the Einstein equation, then
$g$ is called a \emph{Poincar\'e-Einstein metric}.
The Einstein constant must be $-(2n-1)$ for such a metric.
(Restriction to considering \emph{smooth} conformally compact metrics is justified by
the result of Chru\'sciel--Delay--Lee--Skinner \cite{Chrusciel-Delay-Lee-Skinner-05}
because the boundary is odd-dimensional.)

The set of $\SU(n)$-invariant conformal classes on $S^{2n-1}$ can be determined as follows.
Such a class is represented by an $\SU(n)$-invariant Riemannian metric $h$,
and it corresponds to a $K$-invariant positive-definite inner product on $T_qS^{2n-1}$,
where $K$ is the isotropic subgroup of $\SU(n)$ at an arbitrarily fixed point $q\in S^{2n-1}$.
If $n\ge 3$, then $K\cong\SU(n-1)$ decomposes the tangent space irreducibly as $T_qS^{2n-1}=V_1\oplus V_2$,
where $\dim V_1=1$ and $\dim V_2=2n-2$.
As $K$ acts transitively on the set of lines in $V_2$,
$K$-invariant inner products of $V_2$ are unique up to a constant factor,
and hence we may deduce that $h=\lambda_1\theta^2+\lambda_2g_{\mathbb{C}P^{n-1}}$ for some $\lambda_1$, $\lambda_2>0$.
Then $h$ is conformal to $h_c$, where $c=\lambda_1/\lambda_2$.
If $n=2$, this argument fails because $K\cong\SU(1)$ is trivial and $V_2$ is no longer irreducible.

Now we begin the construction of our metrics $g_c$.
As the complex manifold $M$ in the general setting illustrated in the previous section,
we take the complex projective space $\mathbb{C}P^{n-1}$.
The metric $\overline{g}$ is going to be the Fubini--Study metric $g_{\mathbb{C}P^{n-1}}$,
whose Einstein constant $\lambda$ is $2n$.
The line bundle $L$ will be the tautological bundle over $\mathbb{C}P^{n-1}$, thereby
$C$ is naturally identified with the unit sphere $S^{2n-1}$ in $\mathbb{C}^n$.
The connection form $\theta$ on $S^{2n-1}$ is the standard contact form, which is
\begin{equation*}
	\theta=
	\frac{i}{2}\sum_{j=1}^n(z^jd\smash{\conj{z}}^j-\smash{\conj{z}}^jdz^j)
	\Biggr|_{TS^{2n-1}}.
\end{equation*}
The curvature form, regarded as a 2-form on $S^{2n-1}$, is
\begin{equation*}
	F=d\theta
	=i\sum_{j=1}^ndz^j\wedge d\smash{\conj{z}}^j\Biggr|_{TS^{2n-1}},
\end{equation*}
and this is nothing but the pullback of the K\"ahler form of $g_{\mathbb{C}P^{n-1}}$.
Thus we look for Einstein metric of the form
\begin{equation}
	\label{eq:ansatz-for-our-case}
	g=\alpha(r)^2dr^2+\beta(r)^2\theta^2+\gamma(r)^2g_{\mathbb{C}P^{n-1}}
\end{equation}
and they are automatically invariant under the $\SU(n)$-action on the second factor of $I\times S^{2n-1}$.
We take $\Lambda=-(2n-1)$, because it is our normalization of Poincar\'e-Einstein metrics.

Then the polynomial $P$ given by \eqref{eq:P-as-integral} provides an Einstein metric.
Any integration constant can be chosen, but we make the following particular choice
(the idea behind it is that the resulting metrics should be identical to Pedersen's for $n=2$):
\begin{equation}
	\label{eq:P-for-SU-invariant-PE}
	P=P_n
	=r\int_1^r(\lambda c^{-1}t^{-2}(t^2-1)^{n-1}-\Lambda t^{-2}(t^2-1)^n)dt
	=(2n-1)Q_n+2nc^{-1}Q_{n-1}.
\end{equation}
Here $Q_k$ is the polynomial of $r$ of order $2k$ defined by
\begin{equation}
	\label{eq:definition-of-Q}
	Q_k=r\int_1^r t^{-2}(t^2-1)^kdt.
\end{equation}
It is obvious that $Q_k$ is positive for any $r>1$, and hence $r_0=1$ (see \eqref{eq:r-zero} for
the definition of $r_0$).
Therefore \eqref{eq:P-for-SU-invariant-PE} gives an Einstein metric on $(1,\infty)\times S^{2n-1}$.

The polynomial $P_n$ is explicitly given for $n=2$ and $3$ by
\begin{subequations}
\begin{align}
	\label{eq:polynomial-P2}
	P_2&=(r-1)^3(r+3)+4c^{-1}(r-1)^2,\\
	\label{eq:polynomial-P3}
	P_3&=(r-1)^4(r^2+4r+5)+2c^{-1}(r-1)^3(r+3).
\end{align}
\end{subequations}

Under the choice made, $r=r_0=1$ turns out to be just a removable singularity.
In order to see it, the following property of $Q_k$ is crucial.

\begin{lem}
	\label{lem:polynomial-Q}
	The polynomial $Q_k$ defined by \eqref{eq:definition-of-Q} is divisible by $(r-1)^{k+1}$.
	Moreover, if we set
	\begin{equation*}
		Q_k(r)=(r-1)^{k+1}\tilde{Q}_k(r),
	\end{equation*}
	then $\tilde{Q}_k(1)=2^k/(k+1)$.
\end{lem}

\begin{proof}
	It is immediate that $Q_k(1)=Q'_k(1)=0$ and $Q''_k(r)=2k(r^2-1)^{k-1}$
	(this can also be seen as a characterization of $Q_k$).
	Therefore $Q_k$ is divisible by $(r-1)^{k+1}$, and
	\begin{equation*}
		\tilde{Q}_k(1)=\frac{1}{k(k+1)}\cdot\left.\frac{Q''_k(r)}{(r-1)^{k-1}}\right|_{r=1}=\frac{2^k}{k+1}.
		\qedhere
	\end{equation*}
\end{proof}

Let us complete the proof of Theorem \ref{thm:SUn-invariant-fillable}.

\begin{proof}[Proof of Theorem \ref{thm:SUn-invariant-fillable}]
	The metric $g$ on $(1,\infty)\times S^{2n-1}$ is defined
	by \eqref{eq:ansatz-for-our-case}, \eqref{eq:alpha-beta-gamma}, and \eqref{eq:P-for-SU-invariant-PE}.
	We shall use $s=r-1$ in order to track the behavior of the coefficients $\alpha$, $\beta$, $\gamma$ when $r\to 1$.
	Then $P_n$ is a polynomial also in $s$,
	and by \eqref{eq:P-for-SU-invariant-PE} and Lemma \ref{lem:polynomial-Q}, as $s\to 0$,
	\begin{equation*}
		P_n=2nc^{-1}Q_{n-1}+O(s^{n+1})=2^nc^{-1}s^n+O(s^{n+1}).
	\end{equation*}
	On the other hand, $r^2-1=2s+O(s^2)$. Therefore,
	\begin{align*}
		\alpha^2&=(r^2-1)^{n-1}P_n^{-1}=\tfrac{1}{2}cs^{-1}+O(1),\\
		\beta^2&=c^2(r^2-1)^{-n+1}P_n=2cs+O(s^2),\\
		\gamma^2&=c(r^2-1)=2cs+O(s^2),
	\end{align*}
	where each term indicated by $O(1)$ or $O(s^2)$ has a Taylor expansion in $s$.

	Now we perform the following coordinate change
	(the new variable $\rho$ varies within the range $0<\rho<1$):
	\begin{equation}
		\label{eq:coordinate-change-r-to-rho}
		s=r-1=\frac{2c^{-1}\rho^2}{1-\rho^2}.
	\end{equation}
	Then $g$ is expressed as
	\begin{equation}
		\label{eq:metric-in-terms-of-rho}
		g=\tilde{\alpha}^2d\rho^2+\beta^2\theta^2+\gamma^2g_{\mathbb{C}P^{n-1}}
		\qquad\text{with $\tilde{\alpha}^2=\alpha^2\left(\dfrac{ds}{d\rho}\right)^2$},
	\end{equation}
	where
	\begin{equation*}
		\tilde{\alpha}^2=4+O(\rho^2),\qquad
		\beta^2=4\rho^2+O(\rho^4),\qquad
		\gamma^2=4\rho^2+O(\rho^4)
	\end{equation*}
	and the terms indicated by $O(\rho^2)$ and $O(\rho^4)$ are expanded in \emph{even powers} of $\rho$.
	Since $\theta^2+g_{\mathbb{C}P^{n-1}}$ is the round metric on $S^{2n-1}$,
	this implies that $\rho=0$ is just a polar-type singularity.
	Consequently, we obtain a smooth Einstein metric on the unit ball $B^{2n}$, which we now write $g_c$,
	where $\rho$ is regarded as the Euclidean distance from the center.
	
	Interpreted as a metric on $B^{2n}$ in this way, $g_c$ is smooth conformally compact.
	To verify this, note that $r^{-1}$ extends up to the boundary of $B^{2n}$ as a smooth boundary defining function.
	As $r\to\infty$, \eqref{eq:P-for-SU-invariant-PE} implies that $P_n$ is asymptotic to $r^{2n}$ and hence
	\begin{equation}
		\label{eq:behavior-at-infinity}
		\alpha^2\sim r^{-2},\qquad
		\beta^2\sim c^2r^2,\qquad
		\gamma^2\sim cr^2,
	\end{equation}
	and it is easily seen that $r^2\alpha^2$, $r^{-2}\beta^2$, and $r^{-2}\gamma^2$ are all smooth
	up to $\bdry B^{2n}$. Since
	\begin{equation*}
		r^{-2}g_c=r^2\alpha^2d(r^{-1})^2+r^{-2}\beta^2\theta^2+r^{-2}\gamma^2g_{\mathbb{C}P^{n-1}},
	\end{equation*}
	$g_c$ is smooth conformally compact.
	Furthermore, the asymptotic behavior \eqref{eq:behavior-at-infinity} implies that
	\begin{equation*}
		\lim_{r\to\infty}\frac{\beta^2}{\gamma^2}=c,
	\end{equation*}
	and hence the conformal infinity of $g_c$ is given by the conformal class $[h_c]$, where $h_c$ is the
	Berger-type metric \eqref{eq:Berger-metric}.
\end{proof}

When $c=1$, our polynomial $P_n$ is given by $P_n=(r^2-1)^n$ and the metric $g_1$ is
\begin{equation*}
	g_1=\frac{dr^2}{r^2-1}+(r^2-1)g_{S^{2n-1}}
	=\frac{4}{(1-\rho^2)^2}(d\rho^2+\rho^2g_{S^{2n-1}}),
\end{equation*}
where $g_{S^{2n-1}}$ is the round metric on $S^{2n-1}$.
This is the hyperbolic metric.

The construction above recovers the metrics of Pedersen when $n=2$ as follows.
Recall that the expression \eqref{eq:Pedersen-metric} of the metric makes sense for $m^2>-1$.
We take $c$ so that $m^2=c^{-1}-1$.
Then the polynomial $P_2$, given by formula \eqref{eq:polynomial-P2}, becomes
\begin{equation*}
	P_2=16c^{-3}\cdot\frac{\rho^4(1+m^2\rho^4)}{(1-\rho^2)^4}
\end{equation*}
upon the coordinate change \eqref{eq:coordinate-change-r-to-rho}.
Then one can easily check that the metric $g_c$ is exactly \eqref{eq:Pedersen-metric}.

\section{Limiting behavior}
\label{sec:limiting-behavior}

In this section we describe the behavior of the metrics $g_c$, constructed in the previous section,
when the parameter $c$ tends to $\infty$ or $0$.

\subsection{Behavior when $c\to\infty$}

As mentioned in the introduction, Pedersen observed that the metric \eqref{eq:Pedersen-metric} converges
to the complex hyperbolic metric.
In fact, it is clear that \eqref{eq:Pedersen-metric} converges when $c\to\infty$ (i.e., $m^2\to -1$) to
\begin{equation*}
	g
	=\frac{4}{(1-\rho^2)^2}
	\left(\frac{1}{1+\rho^2}d\rho^2
	+\rho^2(1+\rho^2)\theta^2
	+\rho^2(1-\rho^2)g_{\mathbb{C}P^1}\right),
\end{equation*}
and if we put $1-\tilde{\rho}^2=(1-\rho^2)/(1+\rho^2)$ into this formula,
this becomes identical to \eqref{eq:complex-hyperbolic-metric} (with $\rho$ replaced by $\tilde{\rho}$).

Before discussing general $n$, let us examine the case $n=3$ to get a flavor of the situation.
In this dimension the polynomial $P_3$ is given by \eqref{eq:polynomial-P3},
and if we introduce $\rho$ by \eqref{eq:coordinate-change-r-to-rho}, it is
\begin{equation}
	\label{eq:P-for-6-dim}
	P_3=\frac{64c^{-4}\rho^6}{(1-\rho^2)^6}
	\left(1-\frac{1}{2}(1-c^{-1})\rho^2-2(1-c^{-1})\rho^4+\frac{1}{2}(3-5c^{-1}+2c^{-2})\rho^6\right).
\end{equation}
Consequently, the metric $g_c$ is expressed as
\begin{equation}
	\label{eq:family-for-6-dim}
	g_c=\frac{4}{(1-\rho^2)^2}
	\left(\Phi_3^{-1}d\rho^2+\rho^2\Phi_3\theta^2
	+\rho^2(1-\rho^2+c^{-1}\rho^2)g_{\mathbb{C}P^2}\right),
\end{equation}
where
\begin{equation*}
	\Phi_3
	=\frac{1-\frac{1}{2}(1-c^{-1})\rho^2-2(1-c^{-1})\rho^4
	+\frac{1}{2}(3-5c^{-1}+2c^{-2})\rho^6}{(1-\rho^2+c^{-1}\rho^2)^2}.
\end{equation*}
Therefore $g_c$ converges, in the $C^\infty$ topology on any compact subset of $B^6$, to
\begin{equation*}
	g_\infty=\frac{4}{(1-\rho^2)^2}
	\left(\frac{1}{1+\frac{3}{2}\rho^2}d\rho^2+\rho^2(1+\tfrac{3}{2}\rho^2)\theta^2
	+\rho^2(1-\rho^2)g_{\mathbb{C}P^2}\right).
\end{equation*}
This is actually the complex hyperbolic metric \eqref{eq:complex-hyperbolic-metric} up to diffeomorphism action
and a constant factor.
Namely, if we set $1-\tilde{\rho}^2=(1-\rho^2)/(1+\frac{3}{2}\rho^2)$, then
\begin{equation*}
	g_\infty
	=\frac{8}{5}\left(\frac{d\tilde{\rho}^2+\tilde{\rho}^2\theta^2}{(1-\tilde{\rho}^2)^2}
	+\frac{\tilde{\rho}^2g_{\mathbb{C}P^2}}{1-\tilde{\rho}^2}\right)
	=\frac{4}{5}g_{\mathbb{C}H^3}.
\end{equation*}
Thus we get the proof of Theorem \ref{thm:limit-c-infinity} for $n=3$.

Now we consider the general case.

\begin{proof}[Proof of Theorem \ref{thm:limit-c-infinity}]
	In terms of $\rho$, the metric $g_c$ is expressed as \eqref{eq:metric-in-terms-of-rho}.
	We are going to show that it is convergent to
	\begin{equation}
		\label{eq:g-infinity-in-rho}
		g_\infty=\frac{4}{(1-\rho^2)^2}
		\left(\frac{1}{1+a_n\rho^2}d\rho^2
		+\rho^2(1+a_n\rho^2)\theta^2+\rho^2(1-\rho^2)g_{\mathbb{C}P^{n-1}}\right)
	\end{equation}
	in the $C^\infty$ topology on any compact subset of $B^{2n}$, where $a_n=3(n-1)/(n+1)$.
	For this, we need to show that
	\begin{equation*}
		\tilde{\alpha}^2\to\frac{4}{(1-\rho^2)^2}\cdot\frac{1}{1+a_n\rho^2},\qquad
		\rho^{-2}\beta^2\to\frac{4}{(1-\rho^2)^2}\cdot(1+a_n\rho^2),\qquad
		\rho^{-2}\gamma^2\to\frac{4}{1-\rho^2}
	\end{equation*}
	in the $C^\infty$ topology on compact subsets.
	Once this is done, then since $1-\tilde{\rho}^2=(1-\rho^2)/(1+a_n\rho^2)$ makes
	\eqref{eq:g-infinity-in-rho} into $\frac{n+1}{2n-1}g_{\mathbb{C}H^n}$, the proof is complete.

	The formula like \eqref{eq:P-for-6-dim} for general $n$ can be obtained by
	first expressing the polynomial $P_n$ in terms of $s=r-1$
	and then putting \eqref{eq:coordinate-change-r-to-rho} into it.
	Any occurrence of $s$ ends up in a factor $c^{-1}$, and the result will look like
	\begin{equation*}
		P_n=c^{-n-1}p_{n+1}(\rho)+c^{-n-2}p_{n+2}(\rho)+\dots+c^{-2n}p_{2n}(\rho),
	\end{equation*}
	where each $p_j(\rho)$ is independent of $c$. What we need here is the expression of $p_{n+1}(\rho)$.
	By \eqref{eq:P-for-SU-invariant-PE} and Lemma \ref{lem:polynomial-Q},
	\begin{equation*}
		\begin{split}
			p_{n+1}(\rho)
			&=(2n-1)\cdot\frac{2^n}{n+1}\left(\frac{2\rho^2}{1-\rho^2}\right)^{n+1}
			+2n\cdot\frac{2^{n-1}}{n}\left(\frac{2\rho^2}{1-\rho^2}\right)^n\\
			&=\frac{4\rho^2}{(1-\rho^2)^2}\left(\frac{4\rho^2}{1-\rho^2}\right)^{n-1}(1+a_n\rho^2).
		\end{split}
	\end{equation*}
	In addition,
	\begin{equation*}
		r^2-1=\frac{4c^{-1}\rho^2}{1-\rho^2}+\frac{4c^{-2}\rho^4}{(1-\rho^2)^2}
		\qquad\text{and}\qquad
		\frac{dr}{d\rho}=\frac{4c^{-1}\rho}{1-\rho^2}.
	\end{equation*}
	Therefore, by calculating using \eqref{eq:alpha-beta-gamma},
	we obtain the claimed convergence of $\tilde{\alpha}^2$, $\beta^2$, and $\gamma^2$.
\end{proof}

\subsection{Behavior when $c\to 0$}

We first remark that the metric \eqref{eq:limit-metric-c-zero} can also be expressed as
\begin{equation}
	\label{eq:limit-metric-c-zero-in-t}
	g_0=\frac{dt^2}{t^2+b_n}+t^2g_{\mathbb{C}P^{n-1}},
\end{equation}
where $b_n=2n/(2n-3)$, by the coordinate change $t=\sqrt{b_n}\sinh r$.
It suffices to verify the statement of Theorem \ref{thm:limit-c-zero}
for the metric \eqref{eq:limit-metric-c-zero-in-t} instead of \eqref{eq:limit-metric-c-zero}.

Again, we examine the case $n=3$ prior to the general proof.
Let us set $r=c^{-1/2}t+1$, or $s=r-1=c^{-1/2}t$,
and write $g_c=\tilde{\alpha}^2dt^2+\beta^2\theta^2+\gamma^2g_{\mathbb{C}P^2}$
($\tilde{\alpha}$ here is different from that in \eqref{eq:metric-in-terms-of-rho}).
Then
\begin{equation*}
	P_3=s^4(s^2+6s-2)+2c^{-1}s^3(s+4)
	=c^{-3}(t^6+2t^4)+c^{-5/2}(6t^5+8t^3)-2c^{-2}t^4,
\end{equation*}
and $r^2-1=c^{-1}t^2+2c^{-1/2}t$. Therefore the coefficients behave as
\begin{align*}
	\tilde{\alpha}^2=\alpha^2\left(\frac{dr}{dt}\right)^2&=c^{-1}(r^2-1)^2P_3^{-1}\to \frac{1}{t^2+2},\\
	\beta^2&=c^2(r^2-1)^{-2}P_3\to 0,\\
	\gamma^2&=c(r^2-1)\to t^2
\end{align*}
uniformly on $B_{R_1,R_2}=\set{R_1<t<R_2}$, and their derivatives in $t$ are uniformly convergent on $B_{R_1,R_2}$
as well. Hence $g_c$ collapses to
\begin{equation}
	\label{eq:limit-metric-c-zero-6-dim-in-t}
	\frac{1}{t^2+2}dt^2+t^2g_{\mathbb{C}P^2}
\end{equation}
in the sense that $g'_c=\tilde{\alpha}^2dt^2+\gamma^2g_{\mathbb{C}P^2}$
converges to \eqref{eq:limit-metric-c-zero-6-dim-in-t} in the $C^\infty$ topology on $B_{R_1,R_2}$.

The diameter of the fiber of $p\colon(0,\infty)\times S^5\to(0,\infty)\times\mathbb{C}P^2$ is $\pi\beta$.
Note that $(r^2-1)^{-2}s^k$ ($k=4$, $3$, $2$) is bounded on $B_R=\set{0<t<R}$ uniformly in $c$ because
\begin{equation*}
	\frac{s^k}{(r^2-1)^2}=\frac{s^k}{(s^2+2s)^2}
	\le\frac{s^k}{(s^2)^{k-2}(2s)^{4-k}}
	=\frac{1}{2^{4-k}}.
\end{equation*}
From this one can observe that $\beta^2=c^2(r^2-1)^{-2}P_3$ converges to $0$ uniformly on $B_R$.
Thus we have shown Theorem \ref{thm:limit-c-zero} for $n=3$.

The discussion above is easily generalized as follows.

\begin{proof}[Proof of Theorem \ref{thm:limit-c-zero}]
	Let $r=c^{-1/2}t+1$. Then \eqref{eq:P-for-SU-invariant-PE} implies that
	\begin{equation}
		\label{eq:polynomial-P-in-t}
		\begin{split}
			P_n
			&=(2n-1)\cdot\frac{1}{2n-1}c^{-n}t^{2n}+2nc^{-1}\cdot\frac{1}{2n-3}c^{-n+1}t^{2n-2}+c^{-n+1/2}\Psi\\
			&=c^{-n}(t^{2n}+b_nt^{2n-2})+c^{-n+1/2}\Psi,
		\end{split}
	\end{equation}
	where $\Psi$ and its derivatives in $t$ of arbitrary order are uniformly bounded for small $c>0$
	on any $B_{R_1,R_2}$.
	As a consequence, if we write $g_c=\tilde{\alpha}^2dt^2+\beta^2\theta^2+\gamma^2g_{\mathbb{C}P^{n-1}}$,
	then $g'_c=\tilde{\alpha}^2dt^2+\gamma^2g_{\mathbb{C}P^{n-1}}$ converges to
	\eqref{eq:limit-metric-c-zero-in-t} in the $C^\infty$ topology on $B_{R_1,R_2}$.

	One can easily see that $(r^2-1)^{-n+1}s^k$ ($k=2n-2$, $2n-3$, $\dotsc$, $n-1$) is
	uniformly bounded on $B_R$.
	Since $Q_n$ is a linear combination of $s^{2n}$, $s^{2n-1}$, $\dotsc$, $s^{n+1}$ and
	$Q_{n-1}$ is that of $s^{2n-2}$, $s^{2n-3}$, $\dotsc$, $s^n$ by Lemma \ref{lem:polynomial-Q},
	$\beta^2=c^2(r^2-1)^{-n+1}P_n=(2n-1)c^2(r^2-1)^{-n+1}Q_n+2nc(r^2-1)^{-n+1}Q_{n-1}$
	converges to $0$ uniformly on $B_R$, and so does $\pi\beta$, which is the diameter of the fiber of $p$.
\end{proof}

\bibliography{myrefs}

\end{document}